\documentclass[12pt,a4paper]{article}      

\RequirePackage[colorlinks,citecolor=blue,urlcolor=blue]{hyperref}
\usepackage{graphicx}
\usepackage{enumerate}
\usepackage[square,sort,comma,numbers]{natbib}
\usepackage{amssymb, amsmath, amsthm}
\newtheorem{proposition}{Proposition}[section] 

\newtheorem{thm}[proposition]{Theorem}
\newtheorem{lemma}[proposition]{Lemma}
\newtheorem{corollary}[proposition]{Corollary}
\newtheorem{remark}[proposition]{Remark}
\newtheorem{example}[proposition]{Example}

\newcommand{\E}{\ensuremath{{\mathbb E}}}
\newcommand{\Pro}{\ensuremath{{\mathbb P}}}

\begin{document}
\title{A note on the generalized Bernoulli and Euler Polynomials}
\author{Bao Quoc Ta\thanks{\AA bo Akademi University, Department of Mathematics, \small FIN-20500 \AA bo, Finland, email: {tbao}@abo.fi.} }
\date{}
\maketitle

\begin{abstract}
In this paper we use probabilistic methods to derive some results on the generalized Bernoulli and generalized Euler polynomials. Our approach is based on the properties of Appell polynomials associated with uniformly distributed and Bernoulli distributed random variables and their sums.
\end{abstract}
%

\section{Introduction}
We start with recalling the definition and basic properties of Appell polynomials. Let $\xi$ be a random variable with some exponential moments, i.e., $\E(e^{\lambda|\xi|})<\infty$ for some $\lambda>0$. The Appell polynomials $Q_{n}^{(\xi)}, n=0,1,2\dots$ associated with $\xi$ are defined via the expansion
\begin{equation}\label{appell}
\frac{e^{ux}}{\E(e^{u\xi})}=\sum_{n=1}^{\infty}\frac{u^n}{n!}Q_{n}^{(\xi)}(x).
\end{equation}
Clearly, in case $\xi\equiv 0$ it holds
\begin{equation}\label{simplest}
Q_{n}^{(0)}(x)=x^n,\quad n=0,1,2,\dots.
\end{equation}
Notice also that $Q_{0}^{(\xi)}(x)=1$ for all $x$.\\
\indent The Appell polynomials have the following properties (see, e.g., Salminen \citep{Sa})
\begin{description}
\item{(i)} Mean value property: \begin{equation}\label{mean}
\E(Q_{n}^{(\xi)}(\xi+x))=x^n.
\end{equation}
\item{(ii)} Recursive differential equation:
\begin{equation*}\label{recursive}
\frac{d}{dx}Q_{n}^{(\xi)}(x)=nQ_{n-1}^{(\xi)}(x).
\end{equation*}
\item{(iii)} If $\xi_1$ and $\xi_2$ are independent random variables then
\begin{equation*}\label{independent}
Q_{n}^{(\xi_1+\xi_2)}(x+y)=\sum_{k=0}^{n}\binom{n}{k}Q_{k}^{(\xi_1)}(x)Q_{n-k}^{(\xi_2)}(y).
\end{equation*}
Choosing here $\xi_2=0$ and $x=0$ gives 
\begin{equation*}\label{form}
Q_{n}^{(\xi_1)}(y)=\sum_{k=0}^{n}\binom{n}{k}Q_{k}^{(\xi_1)}(0)y^{n-k}.
\end{equation*}
\end{description}

\indent Bernoulli and Euler polynomials are, in fact, Appell polynomials as seen in the following examples.
\begin{example}
{(\it Bernoulli polynomials)}
Let $\theta$ be uniformly distributed random variable on $[0,1]$, i.e., $\theta \sim U[0,1]$. Then 
\begin{equation*}
\frac{e^{ux}}{\E(e^{u\theta})}=\frac{ue^{ux}}{e^{u}-1}=\sum_{n=0}^{\infty}\frac{u^n}{n!}B_{n}(x).
\end{equation*}
The polynomials $x\rightarrow B_n(x), n=0,1,\dots$ are called the Bernoulli polynomials (see also \citep[p 809]{AS}). Using (\ref{mean}) and (\ref{recursive}) we may find the explicit expressions: 
$B_0(x)=1,\quad B_1(x)=x-\frac{1}{2}, \quad B_2(x)=x^2-x+\frac{1}{6}, \dots$
\end{example}
\begin{example}\label{exam-2}
(Euler polynomials) Let $\eta$ be a random variable such that $\Pro(\eta=0)=\Pro(\eta=1)=1/2$, i.e., $\eta\sim Ber(1/2)$. Then 
\begin{equation*}
\frac{e^{ux}}{\E(e^{u\theta})}=\frac{2e^{ux}}{e^{u}+1}=\sum_{n=0}^{\infty}\frac{u^n}{n!}E_{n}(x).
\end{equation*}
The polynomials $x\rightarrow E_n(x), n=0,1,\dots$ are called the Euler polynomials(see \citep[p 809]{AS}) and we have, e.g., $E_0(x)=1,\quad E_1(x)=x-\frac{1}{2},\quad E_2(x)=x^2-x, \dots$
\end{example}
In the next section we will define the generalized Bernoulli and the generalized Euler polynomials. We also give new probabilistic proofs for some of their properties via Appell polynomials. In the third section we derive a new identity between the generalized Bernoulli and the generalized Euler polynomials which extends the results in Cheon \citep{Che} and Srivastava and Pint\'er \citep{S-P}.
\section{Generalized Bernoulli and generalized Euler polynomials}
Recall, e.g., from Luke \citep[p 18]{Luc}  and Erd\'elyi \citep[p 253]{Erd} (see also Comtet \citep[p 227]{comtet}) that for a real or complex number $m$, {\it the generalized Bernoulli polynomials} $B_{n}^{(m)}, n=0,1,\dots$ are defined via
\begin{equation}\label{gen-ber}
\frac{u^m e^{ux}}{(e^{u}-1)^m}=\sum_{n=0}^{\infty}\frac{u^n}{n!}B_{n}^{(m)}(x).
\end{equation}
From (\ref{gen-ber}) it immediately follows
\begin{align}
B_{n}^{(0)}(x)&=x^n,\\
\label{sri-19}
B_{n}^{(m+l)}(x+y)&=\sum_{i=0}^{n}\binom{n}{i}B_{i}^{(m)}(x)B_{n-i}^{(l)}(y),\\
\label{sri-24}
B_{n}^{(m)}(x+y)&=\sum_{i=0}^{n}\binom{n}{i}B_{i}^{(m)}(x)y^{n-i},\\
\label{sri-21}
B_{n}^{(m)}(x+1)-B_{n}^{(m)}(x)&=nB_{n-1}^{(m-1)}(x).
\end{align}

\indent In case $m$ is an integer, we may use a probabilistic approach via Appell polynomials. Indeed, setting $\theta^{(m)}:=\sum_{i=1}^{m}\theta_i \quad\text{and}\quad \theta^{(0)}:=0,$
where $\{\theta_i\}$ is an i.i.d sequence of random variables such that $\theta_i\sim U[0,1]$, it holds
\[\E(e^{u\theta^{(m)}})=\Big(\frac{e^u-1}{u}\Big)^m.\]
Consequently, the  Appell polynomials $Q^{(\theta^{(m)})}_{n}$ associated with $\theta^{(m)}$ are  the generalized Bernoulli polynomials $B_{n}^{(m)}$.\\
\indent We exploit the mean value property (\ref{mean}) to give a proof of formula (\ref{sri-21}) as follows: From (\ref{sri-19}),  (\ref{mean}), and (\ref{sri-24}) we obtain
\begin{equation}\label{p-1}
\E\big(B_{n}^{(m)}(x+\theta_1)\big)=\sum_{i=0}^{n}\binom{n}{i}B_{i}^{(m-1)}(0)\E(B_{n-i}(x+\theta_1))=B_{n}^{(m-1)}(x).
\end{equation}
On the other hand, also from (\ref{sri-24}) 
\begin{align}\label{p-2}
\nonumber
\E\big(B_{n}^{(m)}(x+\theta_1)\big)&=\sum_{i=0}^{n}\binom{n}{i}B_{n-i}^{(m)}(0)\E(x+\theta_1)^{i}\\
\nonumber
&=\sum_{i=0}^{n}\binom{n}{i}B_{n-i}^{(m)}(0)\frac{1}{i+1}[(x+1)^{i+1}-x^{i+1}]\\
&=\frac{1}{n+1}(B_{n+1}^{(m)}(x+1)-B_{n+1}^{(m)}(x)).
\end{align}
Combining (\ref{p-1}) and (\ref{p-2}) gives (\ref{sri-21}).

\begin{remark} (i)
From (\ref{p-1}), by induction, for any positive integer $l\leq m$, we obtain
\begin{equation}\label{for-extend}
\E\Big(B_{n}^{(m)}(x+\sum_{i=1}^{l}\theta_i)\Big)=B_{n}^{(m-l)}(x)
\end{equation}
which coincides with the mean value property (\ref{mean}) in case $m=l$.\\
(ii) For non-integer $m$, there does not exist a random variable $\theta^{(m)}$ such that $\big(\frac{e^u-1}{u}\big)^m$   is the moment generating function of $\theta^{(m)}$. This follows, e.g., from the fact the uniform distribution is not infinitely divisible. Hence, we can connect the generalized Bernoulli polynomials with Appell polynomials only in case $m$ is an integer.
\end{remark}
We can generalize the Euler polynomials similarly as the  Bernoulli polynomials. {\it The generalized Euler polynomials} are defined via (see \citep{Erd})
\begin{equation}\label{gen-eul}
\frac{2^m e^{ux}}{(e^{u}+1)^m}=\sum_{n=0}^{\infty}\frac{u^n}{n!}E_{n}^{(m)}(x),
\end{equation}
and it holds
\begin{align}
\label{norma-2}
E_{n}^{(0)}(x)&=x^n,\\
\label{sri-20}
E_{n}^{(k+l)}(x+y)&=\sum_{i=0}^{n}\binom{n}{i}E_{i}^{(k)}(x)E_{n-i}^{(l)}(y),\\
\label{sri-25}
E_{n}^{(m)}(x+y)&=\sum_{i=0}^{n}\binom{n}{i}E_{i}^{(m)}(x)y^{n-i},\\
\label{sri-22}
E_{n}^{(m)}(x+1)+E_{n}^{(m)}(x)&=2E_{n}^{(m-1)}(x).
\end{align}

\indent In case $m$ is an integer, let $\eta_j, i=1\dots m$ be an i.i.d sequence of random variables such that $\eta_j\sim Ber(1/2)$. The Appell polynomials $Q^{(\eta^{(m)})}_{n}$ associated with the random variable $\eta^{(m)}:=\sum_{j=1}^{m}\eta_j$ are the generalized Euler polynomials $E_{n}^{(m)}(x)$.\\
\indent Formula (\ref{sri-22}) is proved similarly as formula (\ref{sri-21}). It is seen that a formula analogous (\ref{for-extend}) is valid for the generalized Euler polynomials, i.e., 
\begin{equation}\label{for-extend-2}
\E\Big(E_{n}^{(m)}(x+\sum_{j=1}^{l}\eta_j)\Big)=E_{n}^{(m-l)}(x).
\end{equation}
\indent We also note similarly as for the generalized Bernoulli polynomials that the generalized Euler polynomials can be connected with the Appell polynomials only if $m$ is an integer. 
\section{Relationships between the generalized Bernoulli and the generalized Euler polynomials}
In this section we will generalize results in Cheon \citep{Che} and in Srivastava and Pint\'er \citep{S-P}.
Let us introduce the polynomials $Q_{n}^{((m)+(l))}$ obtained from the expansion, for $m,l\in\mathbb{C}$
\begin{equation}\label{equa-1}
\Big(\frac{u}{e^u-1}\Big)^m \Big(\frac{2}{e^u+1}\Big)^l e^{ux}=\sum_{n=0}^{\infty}\frac{u^n}{n!}Q_{n}^{((m)+(l))}(x).
\end{equation}
It holds
\begin{equation}\label{equa-2}
Q_{n}^{((m)+(l))}(x+y)=\sum_{k=0}^{n}\binom{n}{k}B_{k}^{(m)}(x)E_{n-k}^{(l)}(y),
\end{equation}
and
\begin{equation}\label{equa-3}
Q_{n}^{((m)+(l))}(x)=\sum_{k=0}^{n}\binom{n}{k}Q_{k}^{((m)+(l))}(0)x^{n-k}.
\end{equation}
Furthermore, since 
\begin{align*}
\Big(\frac{u}{e^u-1}\Big)^m \Big(\frac{2}{e^u+1}\Big)^l&=\Big[\Big(\frac{u}{e^u-1}\Big)^{m-1} \Big(\frac{2}{e^u+1}\Big)^l\Big]\frac{u}{e^u-1}\\
&=\Big[\Big(\frac{u}{e^u-1}\Big)^m \Big(\frac{2}{e^u+1}\Big)^{l-1}\Big]\frac{2}{e^u+1},
\end{align*}
we have
\begin{align}\label{lem-0}
\nonumber
Q_{n}^{((m)+(l))}(x)&=\sum_{k=0}^{n}\binom{n}{k}Q_{k}^{((m-1)+(l))}(0)B_{n-k}(x)\\
&=\sum_{k=0}^{n}\binom{n}{k}Q_{k}^{((m)+(l-1))}(0)E_{n-k}(x).
\end{align}
\indent In case $m, l$ are integers, let us consider 
$\theta^{(m)}:=\sum_{i=1}^{m}\theta_i\quad\text{and}\quad \eta^{(l)}:=\sum_{j=1}^{l}\eta_j,$
where $\theta_i\sim U[0,1], i=1,\dots, m$ and $\eta_j\sim Ber(1/2), j=1,\dots, l$ are independent. Then it is seen that $Q_{n}^{((m)+(l))}, n=0,1\dots,$ are the Appell polynomials associated with $\theta^{(m)}+\eta^{(l)}$.
\begin{lemma}\label{lemma}
The following decomposition holds for all $m, l\in \mathbb{C}$
\begin{equation}\label{relation-1}
Q_{n}^{((m)+(l))}(x)=Q_{n}^{((m)+(l-1))}(x)-\frac{n}{2}Q_{n-1}^{((m-1)+(l))}(x).
\end{equation}
\end{lemma}
\begin{proof}
In the first equality of (\ref{lem-0}), substitute $x+\theta_1$  instead of $x$, take expectations, use the mean value property (\ref{mean}) and apply (\ref{equa-3}) to obtain
\begin{equation*}\label{lem-1}
\E\big(Q_{n}^{((m)+(l))}(x+\theta_1)\big)=Q_{n}^{((m-1)+(l))}(x).
\end{equation*}
Calculating similarly  as in (\ref{p-2}) we get
\begin{equation*}\label{lem-2}
\E\big(Q_{n}^{((m)+(l))}(x+\theta_1)\big)=\frac{1}{n+1}\Big[Q_{n+1}^{((m)+(l))}(x+1)-Q_{n+1}^{((m)+(l))}(x)\Big].
\end{equation*}
Consequently
\begin{equation}\label{lem-3}
Q_{n}^{((m)+(l))}(x+1)-Q_{n}^{((m)+(l))}(x)=nQ_{n-1}^{((m-1)+(l))}(x).
\end{equation}
Moreover, also by (\ref{lem-0})
\begin{equation}\label{lem-4}
\E\big(Q_{n}^{((m)+(l))}(x+\eta_1)\big)=Q_{n}^{((m)+(l-1))}(x),
\end{equation}
and from (\ref{equa-3}) we have
\begin{align}\label{lem-5}
\nonumber
\E\big(Q_{n}^{((m)+(l))}(x+\eta_1)\big)&=\sum_{k=0}^{n}\binom{n}{k}Q_{k}^{((m)+(l))}(0)\E\big((x+\eta_1)^{n-k}\big)\\
&=\frac{1}{2}\Big[Q_{n}^{((m)+(l))}(x+1)+Q_{n}^{((m)+(l))}(x)\Big].
\end{align}
Combining (\ref{lem-4}) and (\ref{lem-5}) implies
\begin{equation}\label{lem-6}
Q_{n}^{((m)+(l))}(x+1)+Q_{n}^{((m)+(l))}(x)=2Q_{n}^{((m)+(l-1))}(x).
\end{equation}
Subtracting (\ref{lem-6}) and (\ref{lem-3}) completes the proof.
\end{proof}
\begin{remark}
Formulas (\ref{lem-3}) and (\ref{lem-6}) generalize formulas (\ref{sri-21}) and (\ref{sri-22}) respectively.
\end{remark}
Our main formula which connects the generalized Bernoulli and the generalized Euler polynomials is given in the next theorem.
\begin{thm}\label{the-1}
For all $m,l\in \mathbb{C}$, it holds
\begin{equation}\label{relation}
\sum_{k=0}^{n}\binom{n}{k}B_{k}^{(m)}(x)E_{n-k}^{(l-1)}(y)=\sum_{k=0}^{n}\binom{n}{k}\Big[B_{k}^{(m)}(x)+\frac{k}{2}B_{k-1}^{(m-1)}(x)\Big]E_{n-k}^{(l)}(y).
\end{equation}
\end{thm}
\begin{proof}
Using (\ref{equa-2}) it is seen that (\ref{relation-1}) can be developed as follows
\begin{align}
\nonumber
&\sum_{k=0}^{n}\binom{n}{k}B_{k}^{(m)}(x)E_{n-k}^{(l)}(y)\\
\nonumber
&=Q_{n}^{(m)+(l)}(x+y)\\
\nonumber
&=Q_{n}^{(m)+(l-1)}(x+y)-\frac{n}{2}Q_{n}^{(m-1)+(l)}(x+y)\\
\label{c2}
&=\sum_{k=0}^{n}\binom{n}{k}B_{k}^{(m)}(x)E_{n-k}^{(l-1)}(y)-\frac{n}{2}\sum_{k=0}^{n-1}\binom{n-1}{k}B_{k}^{(m-1)}(x)E_{n-k-1}^{(l)}(y)\\
\label{c3}
&=\sum_{k=0}^{n}\binom{n}{k}B_{k}^{(m)}(x)E_{n-k}^{(l-1)}(y)-\sum_{k=0}^{n}\binom{n}{k}\frac{k}{2}B_{k-1}^{(m-1)}(x)E_{n-k}^{(l)}(y),
\end{align}
from which  (\ref{relation}) readily follows.
\end{proof}
Corollaries \ref{corl-1} and \ref{corl-2} below can be found in  Srivastava and Pint\'er \citep{S-P} as Theorem 1 and Theorem 2, respectively.
\begin{corollary}\label{corl-1}
For all $m\in \mathbb{C}$
\begin{equation}\label{S-P-them1}
B_{n}^{(m)}(x+y)=\sum_{k=0}^{n}\binom{n}{k}\Big[B_{k}^{(m)}(x)+\frac{k}{2}B_{k-1}^{(m-1)}(x)\Big]E_{n-k}(y).
\end{equation}
\end{corollary}
\begin{proof}
This follows from (\ref{relation}) by putting $l=1$ and using (\ref{norma-2}) and (\ref{sri-24}).
\end{proof}
 
\begin{corollary}\label{corl-2}
For all $l\in\mathbb{C}$
\begin{equation}\label{relation-3}
E_{n}^{(l)}(x+y)=\sum_{k=0}^{n}\binom{n}{k}\frac{2}{k+1}\Big[E_{k+1}^{(l-1)}(y)-E_{k+1}^{(l)}(y)\Big]B_{n-k}(x).
\end{equation}
\end{corollary}
\begin{proof}
Notice that in the step from (\ref{c2}) to (\ref{c3}) in the proof of Theorem \ref{the-1} we have the identity
 \[\sum_{k=0}^{n}\binom{n}{k}\frac{k}{2}B_{k-1}^{(m-1)}(x)E_{n-k}^{(l)}(y)=\frac{n}{2}\sum_{k=0}^{n-1}\binom{n-1}{k}B_{k}^{(m-1)}(x)E_{n-k-1}^{(l)}(y).\]
Hence,
\[\sum_{k=0}^{n}\binom{n}{k}\frac{k}{2}B_{k-1}^{(0)}(x)E_{n-k}^{(l)}(y)=\frac{n}{2}E_{n-1}^{(l)}(x+y).\]
From this and (\ref{relation}) with $m=1$ we now have
  \begin{align*}\label{cc1}
E_{n-1}^{(l)}(x+y)=\frac{2}{n}\sum_{k=0}^{n}\binom{n}{k}\Big[E_{k}^{(l-1)}(y)-E_{k}^{(l)}(y)\Big]B_{n-k}(x).
\end{align*}
Changing $n$ to $n+1$ and noting that $E_{0}^{(l-1)}(y)=E_{0}^{(l)}(y)=1$ we have
\begin{align*}
E_{n}^{(l)}(x+y)=\frac{2}{n+1}\sum_{k=1}^{n+1}\binom{n+1}{k}\Big[E_{k}^{(l-1)}(y)-E_{k}^{(l)}(y)\Big]B_{n+1-k}(x),
\end{align*}
which implies (\ref{relation-3}).
\end{proof}
\indent We recall also the following formula due to Cheon pp{Che} 
\begin{equation}\label{che}
B_n(y)=\sum_{k=0, k\neq 1}^{n}\binom{n}{k}B_{k}(0)E_{n-k}(y),
\end{equation}
which is now obtained from (\ref{S-P-them1}) by taking $x=0$ and $m=1$. The next result is derived in Srivastava and Pint\'er \citep{S-P}. We conclude this paper by giving a new proof for this.
\begin{proposition}
Formula (\ref{che}) is equivalent to 
\begin{equation}\label{s-p}
2^nB_{n}(x/2)=\sum_{k=0}^{n}\binom{n}{k}B_{k}(0)E_{n-k}(x).
\end{equation}
\end{proposition} 
\begin{proof}
From(\ref{relation-1}) putting $m=1, l=1$,  we have
\begin{equation*}\label{new-rela}
Q_{n}^{(\theta^{(1)}+\eta^{(1)})}(x)=B_n(x)-\frac{n}{2}E_{n-1}(x).
\end{equation*}
On the other hand
\[\frac{e^{ux}}{\E(e^{u(\theta^{(1)}+\eta^{(1)})})}=\frac{2ue^{ux}}{e^{2u}-1}=\sum_{n=0}^{\infty}\frac{u^n}{n!}2^nB_{n}(x/2),\]
and, hence,
\begin{equation*}
Q_{n}^{(\theta^{(1)}+\eta^{(1)})}(x)=2^nB_{n}(x/2).
\end{equation*}
So we obtain
\begin{equation*}
B_n(x)-\frac{n}{2}E_{n-1}(x)=2^nB_{n}(x/2),
\end{equation*}
which implies the equivalence of (\ref{che}) and (\ref{s-p}). 
\end{proof}
 \section*{Acknowledgements}
The author would like to thank Professor Paavo Salminen for valuable comments and discussions which improved this paper.

\end{document}